\documentclass[10pt]{article}
\usepackage[dvips]{color}
\usepackage{epsfig}
\usepackage{float,amsthm,amssymb,amsfonts}
\usepackage{ amssymb,amsmath,graphicx, amsfonts, latexsym}
\begin{document}
\theoremstyle{plain}
\newtheorem{theorem}{{\bf Theorem}}[section]
\newtheorem{corollary}[theorem]{Corollary}
\newtheorem{lemma}[theorem]{Lemma}
\newtheorem{proposition}[theorem]{Proposition}
\newtheorem{remark}[theorem]{Remark}

\theoremstyle{definition}
\newtheorem{defn}{Definition}
\newtheorem{definition}[theorem]{Definition}
\newtheorem{example}[theorem]{Example}
\newtheorem{conjecture}[theorem]{Conjecture}

\def\im{\mathop{\rm Im}\nolimits}
\def\dom{\mathop{\rm Dom}\nolimits}
\def\rank{\mathop{\rm rank}\nolimits}
\def\nullset{\mbox{\O}}
\def\ker{\mathop{\rm ker}\nolimits}
\def\implies{\; \Longrightarrow \;}

\def\GR{{\cal R}}
\def\GL{{\cal L}}
\def\GH{{\cal H}}
\def\GD{{\cal D}}
\def\GJ{{\cal J}}

\def\set#1{\{ #1\} }
\def\z{\set{0}}
\def\Sing{{\rm Sing}_n}
\def\nullset{\mbox{\O}}

\title{Regularity and Green's relations for the semigroup of partial
contractions of a finite chain }
\author{\bf B. Ali, A. Umar and M. M. Zubairu \footnote{Corresponding Author. ~~Email: $aumar@pi.ac.ae$} \\
\it\small Department of Mathematics, Nigerian Defence Academy, Kaduna\\
\it\small  \texttt{bali@nda.edu.ng}\\[3mm]
\it\small Department of Mathematics, The Petroleum Institute, Sas Nakhl,\\
\it\small Khalifa University of Science and Technology, P. O. Box 2533 Abu Dhabi, UAE\\
\it\small  \texttt{aumar@pi.ac.ae}\\[3mm]
\it\small  Department of Mathematics, Bayero  University Kano, P. M. Box 3011 Kano Nigeria\\
\it\small  \texttt{mmzubairu.mth@buk.edu.ng}\\
}
\date{\today}
\maketitle\

\begin{abstract}
 Let $[n]=\{1,2,\ldots,n\}$ be a finite chain and let  $\mathcal{P}_{n}$  be the semigroup of partial transformations on $[n]$.  Let $\mathcal{CP}_{n}=\{\alpha\in \mathcal{P}_{n}: ( for ~all~x,y\in \dom~\alpha)~|x\alpha-y\alpha|\leq|x-y|\}$, then $\mathcal{CP}_{n}$ is a  subsemigroup of $\mathcal{P}_{n}$. In this paper, we give a necessary and sufficient condition for an element in
 $\mathcal{P}_{n}$ to be regular and characterize all the Green's equivalences on the semigroup $\mathcal{CP}_{n}$.
 \end{abstract}

\section{Introduction and Preliminaries}
 Let $[n]=\{1,2, \ldots ,n\}$ be a finite chain, a map $\alpha$ which has domain and range both subsets of $[n]$ is said to be a \emph{(partial) transformation}.  The collection of all partial transformations of $[n]$ is known as the \emph{semigroup of partial transformations}, usually denoted by $\mathcal{P}_{n}$. A map $\alpha\in \mathcal{P}_{n}$ is said to be \emph{order preserving} (resp., \emph{order reversing}) if  (for all $x,y \in \dom~\alpha$) $x\leq y$ implies $x\alpha\leq y\alpha$ (resp. $x\alpha\geq y\alpha$); is \emph{order decreasing} if (for all $x\in \dom~\alpha$) $x\alpha\leq x$; is an \emph{isometry} (i. e., \emph{ distance preserving}) if (for all $x,y \in \dom~\alpha$) $|x\alpha-y\alpha|=|x-y|$;   a \emph{contraction} if (for all $x,y \in \dom~\alpha$) $|x\alpha-y\alpha|\leq |x-y|$. Let

 \begin{equation} \label{eq1}\mathcal{CP}_{n}=\{\alpha\in \mathcal{P}_{n}:( for ~all~x,y\in \dom~\alpha)~ |x\alpha-y\alpha|\leq |x-y|\}  \end{equation}
 and

 \begin{equation} \label{eq2}\mathcal{OCP}_{n}=\{\alpha\in \mathcal{CP}_{n}: ( for ~all~x,y\in \dom~\alpha)~x\leq y ~ implies ~ x\alpha\leq y\alpha\},\end{equation}
 be the subsemigroups of \emph{partial contractions} and of \emph{order preserving partial contractions} of $[n]$, respectively. A general study of these semigroups was first proposed in 2013  by Umar and AlKharousi \cite{af} (a research proposal supported by a grant from The Research Council of Oman - TRC). Umar and AlKharousi \cite{af} proposed among other things, notations for these semigroups and their subsemigroups as such we maintain the same notations in this paper.
  For standard and basic concepts in semigroup theory, we refer the reader to Howie  \cite{howi} and Higgins \cite{ph}.

  Regularity and Green's relations on the semigroup $\mathcal{P}_{n}$ and its various subsemigroups have been studied by many authors, see for example, \cite{p,Maz,howi,kd,gg,ggg,su,zou,pe,sp,sz, ua, py}. It is now the case that whenever one encounters a new class of semigroups, the first question usually raised is about its Green's equivalences. Recently,  Zhao and  Yang \cite{py} characterized regular elements and all the Green's equivalences on $\mathcal{CPO}_{n}$, where they refer to our``contractions" as``compressions". However, so far, nothing has been done on regularity and Green's relations for the new semigroup $\mathcal{CP}_{n}$. In this paper, in Section 2, we give necessary and sufficient conditions for an element in $\mathcal{CP}_{n}$ to be regular and in Section 3, we describe all the Green's equivalences. Most of the results concerning regularity and Green's relations of subsemigroups of $\mathcal{CP}_{n}$ can be deduced from the results obtained in this paper. We have demonstrated this assertion in Section 4 by deducing the results of Zhao and  Yang \cite{py}. For the remainder of this section we prove some preliminary results that will be needed later.

 Let $\alpha$ be in $\mathcal{CP}_{n}$ and let  $\dom\,\alpha$, $\im~\alpha$ and  $h~(\alpha)$ denote the domain of $\alpha$, image  of $\alpha$ and $|\im~\alpha|$, respectively. For $\alpha,\beta \in \mathcal{CP}_{n}$,  the composition of $\alpha$ and $\beta$ is defined as $x(\alpha \circ \beta) =((x)\alpha)\beta$ for any $x$ in $\dom~\alpha$.  Without ambiguity, we shall be using the notation $\alpha\beta$ to denote $\alpha \circ \beta$.

  Next, let $A$, $B$ be any nonempty subsets of  $[n]$. $A$ is said to \emph{precede} $B$  written as $A\prec B$  if  $a< b$ for arbitrary $a\in A$, $b\in B$  or $\min A< \min B$. Thus, if  $a< b$ for arbitrary $a\in A$, $b\in B$ then $A\prec B$ coincides with the natural partial ordering and  we can write $A< B$ instead of ($A\prec B$) otherwise we maintain the notation  ($A\prec B$).

   Further, given any transformation $\alpha$ in $\mathcal{P}_{n}$, domain of $\alpha$  is partitioned into $p-blocks$ by the relation $\ker~\alpha=\{(x,y)\in [n]\times [n]: x\alpha=y\alpha\}$, i. e., if

    \begin{equation}\label{1}
    \alpha=\left( \begin{array}{cccc}
                           A_{1} & A_{2} & \ldots & A_{p} \\
                           x_{1} & x_{2} & \ldots & x_{p}
                         \end{array}
   \right)\in \mathcal{P}_{n}  ~~  (1\leq p\leq n),
    \end{equation}

     \noindent then $A_{i}$ ($1\leq i\leq p$) are equivalence classes under the relation $\ker~\alpha$. Thus, $a_{i}\alpha=x_{i}$ for all $a_{i}\in A_{i}$ ($1\leq i\leq p$). The collection of all the equivalence classes of the relation $\ker~\alpha$, is the partition of the domain of  $\alpha$, and is denoted by $\textbf{Ker}~\alpha$, i. e., $\textbf{Ker}~\alpha=\{A_{1}, A_{2}, \ldots A_{p}\}$ and $\dom~\alpha=A_{1}\cup A_{2}\cup\ldots A_{p}$ where $(p\leq n)$. We now have the following lemma.

\begin{lemma}\label{min} Let $A$ and $B$ be any disjoint subsets of $[n]$ then there exist $a^{\prime}\in A$ and  $b^{\prime}\in B$ such that $\left|a^{\prime}-b^{\prime}\right|\leq \left|a-b\right|$ for all $a\in A$, $b\in B$.
\end{lemma}
\begin{proof}   Define a set as $(A-B)^{\prime}=\{\left|a-b\right|:a\in A, b\in B\}$, clearly $(A-B)^{\prime}$ is a subset of $\mathbb{N}$. The result now follows by the well ordering property of $\mathbb{N}$.
\end{proof}

  Now, let $\alpha$ be as defined in \eqref{1} and  $\textbf{Ker}~\alpha=\left(A_{i}\right)_{i\in[p]}=\{A_{1}\prec A_{2}\prec \ldots \prec A_{p}\}$ be the partition of $\dom~\alpha$  ordered with the relation $\prec$. A subset $T_{\alpha}$ of $[n]$ is said to be a \emph{transversal} of the partition $\textbf{Ker}~\alpha$ if $|T_{\alpha}|=p$, and $|A_{i}\cap T_{\alpha}|=1$ ($1\leq i\leq p$). A transversal  $T_{\alpha}$  is said to be \emph{relatively convex} if for all $x,y\in T_{\alpha}$ with $x\leq y$ and if $x\leq z\leq y$ ($z\in \dom~\alpha$), then $z\in T_{\alpha}$. Notice that every convex transversal is necessarily relatively convex but not \emph{vice-versa}.

A transversal $T_{\alpha}$ is said to be \emph{admissible} if  for any $t_i, t_j\in T_{\alpha}=\{t_{i}: ~t_{i}\in A_{i}, ~1\leq i\leq p\}$, $\left|t_i-t_j\right|\leq\left|a_{i}-a_{j}\right|$ for all $a_{i}\in A_{i}$, $a_{j}\in A_{j}$ ($i,j\in\{1,2,\ldots,p\}$).
In other words, a transversal $T_{\alpha}$ is admissible if and only if the map $A_{i}\mapsto t_{i}$  ($t_{i}\in T_{\alpha},\, i\in\{1,2,\ldots,p\}$) is a contraction. Notice that every convex transversal
is admissible but not \emph{vice-versa}.

    For the purpose of illustration, consider the following transformations: \\$\alpha_{1}= \left(\begin{array}{cccc}
                                       \{1, 2, 10,  23\} & \{4,  12\} & \{6, 14\} & \{7, 16, 17 \} \\
                                       8 & 6 & 4 & 3
                                     \end{array}
 \right)$,\\ $\alpha_{2}=\left(\begin{array}{ccc}
                                       \{1,5,30\} & \{2,12,10\} & \{4,16\}  \\
                                       8 & 7 & 9
                                     \end{array}
 \right)$, $\alpha_{3}=\left(\begin{array}{ccc}
                                       \{1,7,21\} & \{2,8,20\} & \{3,9,19\}  \\
                                       3 & 4 & 5
                                     \end{array}
 \right)$.

The  partition $\textbf{Ker}~\alpha_{1}$ has an admissible transversal $\{2,4,6,7\}$. Next, for the transformation $\alpha_{2}$, none of the transversals of $\textbf{Ker}~\alpha_{2}$ is admissible.

Now consider $\alpha_{3}$. The transversals $\{1,2,3\}$, $\{7,8,9\}$ and $\{19,20,21\}$ are all admissible and convex.

 \begin{remark}\label{minn} \emph{We observe the following:}

\begin{description}
  \item[(i)] \emph{Every  convex transversal  is an admissible  transversal,  but the converse is not true};
  \item[(ii)] \emph{ Every partition $\textbf{Ker}~\alpha$, of $\dom~\alpha$ in $\mathcal{CP}_{n}$ of height 2 has an admissible transversal. This follows from Lemma\eqref{min};}
  \item[(iii)] \emph{Every admissible transversal is relatively convex}.
\end{description}

 \end{remark}

 Next, we have the following lemma:

\begin{lemma}\label{p1} For $n\geq 4$, let $\alpha\in \mathcal{CP}_{n}$  be such that there exists $k\in\{2,\ldots,p-1\}$ ($3\leq p\leq n$) and $|A_{k}|\geq 2$.  If $A_{i}<A_{j}$ ($i<j$) for all $i,j \in  \{1,2,\ldots,p\}$  then the partition $\textbf{Ker}~\alpha=\{ A_{1}, A_{2},\ldots, A_{p}\}$ of $\dom~\alpha$ has no relatively convex transversal.
 \end{lemma}

 \begin{proof}
 Let $\textbf{Ker}~\alpha=\{ A_{1}, A_{2},\ldots,A_{k-1}, A_{k},A_{k+1},\ldots, A_{p}\}$ be  partition of $\dom~\alpha$ such that $|A_{k}|\geq 2$ where $2\leq k\leq p-1$. Suppose $A_{i}<A_{j}$ ($i<j$) for all $i,j \in  \{1,2,\ldots,p\}$.

  Suppose by way of contradiction that $\textbf{Ker}~\alpha$ has a relatively convex transversal\\ $T_{\alpha}=\{t_{1}, t_{2},\ldots,t_{k-1}, t_{k},t_{k+1}, \ldots, t_{p} \}$  ($t_{i}\in A_{i}$ for all $1\leq i\leq p$). Now since $|A_{k}|\geq 2$, it means that there exists $a_{k}\in A_{k}$ with $a_{k}\neq t_{k}$ and $a_{k}\not\in T_{\alpha}$. Suppose $a_{k}< t_{k}$. Notice that every element in $A_{k-1}$ is less than every element in $A_{k}$, in particular $t_{k-1}<a_{k}< t_{k}$. This contradicts the fact that $T_{\alpha}$ is relatively convex. On the other hand, suppose $t_{k}<a_{k}$. Notice also that $A_{k}<A_{k+1}$, thus $t_{k}<a_{k} <t_{k+1}$. This also contradicts the fact that $T_{\alpha}$ is relatively convex and hence the result follows.

 \end{proof}
 \begin{corollary}\label{p2} For $n\geq 4$, let $\alpha\in \mathcal{ORCP}_{n}$ be such that (for $p\geq 3$) there exists $k\in\{2,\ldots,p-1\}$ and $|A_{k}|\geq 2$. Then the partition $\textbf{Ker}~\alpha=\{ A_{1}, A_{2},\ldots, A_{p}\}$ of $\dom~\alpha$ has no relatively convex transversal.
\end{corollary}

 \begin{lemma}\label{p3}
 Let $\alpha\in \mathcal{CP}_{n}$ be such that $A_{i}<A_{j}$ for all $i<j$ in $\{1,2,\ldots,p\}$ (for
 $ p\geq 3$).  If  $|A_{i}|=1$  for all $2\leq i\leq p-1$ then the partition $\textbf{Ker}~\alpha$ of $\dom~\alpha$ has an \emph{admissible}  transversal $T_{\alpha}$.
  \end{lemma}

 \begin{proof} Take $T_{\alpha}=\{a_{i}: 2\leq i\leq p-1\}\cup \{\max A_{1},~ \min A_{p}\}$. Clearly,  $\theta=\left( \begin{array}{ccccc}
                           A_{1} & a_{2} & \ldots & a_{p-1} & A_{p} \\
                           \max A_{1} & a_{2} & \ldots & a_{p-1} & \min A_{p}
                         \end{array}
   \right)$ is a contraction. Hence $T_{\alpha}$ is admissible.

 \end{proof}

  A map $\alpha\in \mathcal{P}_{n}$ is said to be an \emph{isometry} if and only if $\left|x\alpha-y\alpha\right|=\left|x-y\right|$ for all $x,y\in \dom~\alpha$. If we consider $\alpha$ as defined in \eqref{1}, then $\alpha$ is an \emph{isometry} if and only if $\left|x_{i}-x_{j}\right|=\left|a_{i}-a_{j}\right|$ for all $a_{i}\in A_{i}$ and $a_{j}\in A_{j}$ ($i,j\in\{1,2,\ldots,p\}$). Notice that this forces the blocks $A_{i}$ ($i=1,\ldots,p$) to be singletons, because $\alpha$ is one$-$one. In other words $\alpha$ is an \emph{isometry} if and only if $\dom~{\alpha}=\{a_{i}: 1\leq i\leq p\}=\{x_{i} + e: 1\leq i\leq p\} = (\dom~{\alpha})\alpha + e$ (called a \emph{translation}) or $\dom~{\alpha}=\{a_{i}: 1\leq i\leq p\}=\{x_{p-i+1} + e: 1\leq i\leq p\} = (\dom~{\alpha})\alpha + e$  (called a \emph{reflection}) for some integer $e$.

 Now let $\alpha \in \mathcal{P}_{n}$ be as defined in \eqref{1} ($1\leq p\leq n$). Then we have by the definition of contraction the following lemma:

\begin{lemma}\label{con} An element $\alpha$ in $\mathcal{P}_{n}$ is a contraction if and only if $\left|x_{i}-x_{j}\right|\leq \left|a-b\right|$ for all $a\in A_{i}$ and $b\in A_{j}$  ($i,j\in\{1,2,\ldots,p\}$).

\end{lemma}
The next lemma gives a characterization of contractions with an admissible transversal.

\begin{lemma}\label{connn} Let $\alpha\in \mathcal{P}_{n}$ be such that $\textbf{Ker}~\alpha$ has an admissible transversal, $T_{\alpha}$. Then $\alpha$ is a contraction if and only if $\left|x_{i}-x_{j}\right|\leq\left|t_{i}-t_{j}\right|$ for all $t_{i}, t_{j}\in T_{\alpha}$ for all $i,j\in\{1,2,\ldots,p\}$.
 \end{lemma}
 \begin{proof} Suppose $\alpha\in \mathcal{P}_{n}$ is such that $\textbf{Ker}~\alpha$ has an admissible transversal, $T_{\alpha}$. Further, suppose that $\alpha$ is a contraction. Then by Lemma\eqref{con}, $\left|x_{i}-x_{j}\right|\leq \left|a_{i}-a_{j}\right|$ for all $a_{i}\in A_{i}$ and $a_{j}\in A_{j}$ (for all  $i,j\in\{1,2,\ldots,p\}$). Thus in particular, $\left|x_{i}-x_{j}\right|\leq\left|t_{i}-t_{j}\right|$.

 Conversely, suppose $\left|x_{i}-x_{j}\right|\leq\left|t_{i}-t_{j}\right|$ for all $t_{i}, t_{j}\in T_{\alpha}$ for some $T_{\alpha}$. Notice that  $\left|t_{i}-t_{j}\right|\leq\left|a_{i}-a_{j}\right|$ for all $a_{i}\in A_{i}$ and $a_{j}\in A_{j}$ ( for all $i,j\in\{1,2,\ldots,p\}$). Thus $\left|x_{i}-x_{j}\right|\leq \left|t_{i}-t_{j}\right| \leq\left|a_{i}-a_{j}\right|$. The result follows from Lemma\eqref{con}.

 \end{proof}
 \begin{lemma}\label{tofa} Let $\alpha\in \mathcal{CP}_{n}$ and let $A$ be a convex subset of $\dom~\alpha$. Then $A\alpha$ is convex.
 \end{lemma}
 \begin{proof}
 Let $\alpha\in \mathcal{CP}_{n}$ such that $A\subseteq \dom~\alpha$ is convex. Suppose by way of contradiction that $A\alpha$ is not convex. That is to say, there exist $x,~z\in A\alpha$ with $x<z<y$ for some $z\in [n]\setminus A\alpha$. Let $(z-1]$ and $[z+1)$ be the lower and upper saturations of $z-1$ and $z+1$, respectively. Note that, $x\in (z-1]$ and $y\in [z+1)$. Also note that, $(z-1]\alpha^{-1}\neq \dom~\alpha\neq [z+1)\alpha^{-1}$, but $(z-1]\alpha^{-1}\cup [z+1)\alpha^{-1}=\dom~\alpha$. If $(z-1]\alpha^{-1}$ is convex then, since $(z-1]\alpha^{-1}\neq \dom~\alpha$, there exist either (i) an element $a$ in $(z-1]\alpha^{-1}$ and $a+1$ in $(z+1]\alpha^{-1}$; or (ii) $a$ in $(z-1]\alpha^{-1}$ and $a-1$ in $(z-1]\alpha^{-1}$.
Case i: Clearly $a\alpha\leq z-1$ and $(a+1)\alpha\geq z+1$ so that, $$ 2\leq (a+1)\alpha-a\alpha=\left|(a+1)\alpha-a\alpha\right|\leq \left|(a+1)-a\right|=1,$$
which is a contradiction.
Case ii: it is clear that $a\alpha\leq z-1$ and $(a-1)\alpha\geq y+1$ so that $$2\leq (a-1)\alpha-a\alpha=\left|(a-1)\alpha-a\alpha\right|\leq\left|(a-1)-a\right|=1,$$ which is another contradiction.

Now if $(z-1]\alpha^{-1}$ is not convex then there exists  $a\in (z-1]\alpha^{-1}$ and either $a+1\in (z+1]\alpha^{-1}$ or $a-1\in (z+1]\alpha^{-1}$. In the former, we see that $a\alpha\leq z-1$ and $(a+1)\alpha\geq z+1$. Therefore, $$ 2\leq (a+1)\alpha-a\alpha=\left|(a+1)\alpha-a\alpha\right|\leq \left|(a+1)-a\right|=1,$$
which is a contradiction. In the letter, we see that $a\alpha\leq z-1$ and $(a-1)\alpha\geq z+1$. Thus, $$2\leq (a-1)\alpha-a\alpha=\left|(a-1)\alpha-a\alpha\right|\leq\left|(a-1)-a\right|=1,$$ which is another contradiction. Hence the result follows.

 \end{proof}
\begin{corollary}[\cite{adu}, lemma1.2]\label{tc}  Let $\alpha\in \mathcal{CT}_{n}$ be such that $|\im~\alpha|=p$. Then $\im~\alpha$ is convex.
\end{corollary}
\begin{proof} Let $\alpha\in \mathcal{CT}_{n}$.  Notice that $\dom~\alpha=[n]$ is convex. Thus, by Lemma\eqref{tofa} $[n]\alpha=\im~\alpha$ is convex.
\end{proof}

Next, we have
 \begin{lemma}\label{8} Let $\alpha\in \mathcal{P}_{n}$ be as defined in \eqref{1}  $(3\leq p\leq n)$, such that $\textbf{Ker}~\alpha$ has a convex transversal. Then  $\alpha$  is a contraction if and only if     $T_{\alpha}=(T_{\alpha})\alpha + e$, for some $e\in \mathbb{Z}$.
\end{lemma}
\begin{proof} Suppose $\alpha$ is a contraction whose $\textbf{Ker}~\alpha$ has a convex transversal, $T_{\alpha}$. Then by Lemma\eqref{tofa} we see that $(T_{\alpha})\alpha= \im\,\alpha$ is convex and so
$T_{\alpha}=(T_{\alpha})\alpha + e$, for some $e\in \mathbb{Z}$.

 Conversely, suppose $\textbf{Ker}~\alpha$ has a convex transversal $T_{\alpha}$ such that $T_{\alpha}=(T_{\alpha})\alpha + e$ for some $e\in \mathbb{Z}$. Take $i<j$ with $a_{i}\in A_{i}$ and $a_{j}\in A_{j}$ ($i,j\in\{1,2,\ldots,p\}$) then $$\left|a_{i}-a_{j}\right|\geq \left|t_{i}-t_{j}\right|=\left|(t_{i}\alpha + e)-(t_{j}\alpha + e)\right|=\left|t_{i}\alpha-t_{j}\alpha\right|=\left|x_{i}-x_{j}\right|$$ or
 $$\left|a_{i}-a_{j}\right|\geq \left|t_{i}-t_{j}\right|=\left|(t_{p-i+1}\alpha + e)-(t_{p-j+1}\alpha + e)\right|=\left|t_{p-i+1}\alpha-t_{p-j+1}\alpha\right|=\left|x_{i}-x_{j}\right|.$$
  Thus, the result follows from Lemma\eqref{con}.

\end{proof}

\section{Regularity of elements in $\mathcal{CP}_{n}$}
An element $a$ in a semigroup $S$ is said to be \emph{regular} if and only if there exists $a^{\prime}\in S$ such that $a=aa^{\prime}a$, if every element of $S$ is regular then the semigroup $S$ is said to be a \emph{regular} semigroup. Many transformation semigroups were shown to be regular or their regular elements have been characterized \cite{gg,ggg,zou,pe, sz,ua, py}. In this section we investigate the regular elements of $\mathcal{CP}_{n}$ where we give a sufficient and necessary condition for an element in the semigroup $\mathcal{CP}_{n}$ and some of its subsemigroups to be regular. Recall that Zhao and Yang \cite{py} characterized regular elements in the semigroup of order preserving partial contractions of a finite chain $\mathcal{OCP}_{n}$, but their characterization depends heavily on order preservedness.
Therefore, their characterization of regular elements would not hold in the more general semigroup $\mathcal{CP}_{n}$. To see this consider $\alpha_{1}=\left(\begin{array}{cc}
                                       \{2\} & \{1,3,5\}   \\
                                       1& 2
                                     \end{array}
 \right)\in \mathcal{CP}_{5}$. Denote $A_{1}=\{2\}$ and $A_{2}=\{1,3,5\}$, but  $A_{1}<A_{2}$ or $A_{2}<A_{1}$ does not hold, as such we cannot apply the condition of  Theorem 2.2 in \cite{py}. Similarly for $\alpha_{2}=\left(\begin{array}{ccc}
                                       \{2\} & \{1,3,5\} & \{4\}  \\
                                       1 & 2 & 3
                                     \end{array}
 \right)\in \mathcal{CP}_{10}$ the conditions of  Theorem 2.2 are not satisfied.  Thus, we cannot conclude by Theorem 2.2 (2) and (3) in \cite{py} that  $\alpha_{1}$ and $\alpha_{2}$ are (or are not)  regular, respectively. This is due to the fact that $A_{i}<A_{j}$  for all $i,j\in \{1,2,\ldots,p\}$ ($i<j$) does not hold generally for the elements of $\mathcal{CP}_{n}$. But we shall see in this section that these elements are  regular.

Now we have the main result of this section:

\begin{theorem}\label{reeg} Let $\alpha$ and $\mathcal{CP}_{n}$ be  as defined in \eqref{1} and \eqref{eq1}, respectively, where $\alpha\in \mathcal{CP}_{n}$. Then $\alpha$ is regular if and only if  there exists an admissible transversal $T_{\alpha}$ of $\textbf{Ker}~\alpha$ such that $\left|t_{j}-t_{i}\right|=\left|t_{j}\alpha-t_{i}\alpha\right|$ for all $t_{j}, t_{i}\in T_{\alpha}$ ($i,j\in \{1,2,\ldots,p\}$). Equivalently, $\alpha$ in $\mathcal{CP}_{n}$ is regular if and only if there exists an admissible transversal $T_{\alpha}$, such that the map $t_{i}\mapsto x_{i}$  ($i\in\{1,2,\ldots, p\}$) is an isometry.
\end{theorem}
\begin{proof} Let $\alpha\in \mathcal{CP}_{n}$ be a regular element. Then there exists $\gamma\in \mathcal{CP}_{n}$ such that $\alpha=\alpha\gamma\alpha$. Thus given any $t\in\{1,2,\ldots,p\}$, $x_{t}=A_{t}\alpha=(A_{t}\alpha)\gamma\alpha=(x_{t}\gamma)\alpha$, i. e., $x_{t}\gamma\in A_{t}$ for all $1\leq t\leq p$.
 Now suppose by way of contradiction that for all admissible transversals $T_{\alpha}$ of $\textbf{Ker}~\alpha$, there exist $t_{i}, t_{j}\in T_{\alpha}$ for some $i,j\in \{1,2,\ldots,p\}$  such that \begin{equation}\label{reg} \left|t_{i}\alpha-t_{j}\alpha\right|<\left|t_{i}-t_{j}\right|. \end{equation}

 Let $\{C_{1}, C_{2},\ldots,C_{m}\}$ ($m\leq n$) be the kernel classes of $\ker~\gamma$ arrange in such a way that $A_{k}\alpha\in C_{k}$ (i. e., $x_{k}\in C_{k}$) for $k=1,2,\ldots,p\leq m$. Since $\gamma\in \mathcal{CP}_{n}$, then   $|x_{i}\gamma-x_{j}\gamma|\leq|x_{i}-x_{j}|$ where $x_{i}\in C_{i}$, $ x_{j}\in C_{j}$. Also given any $x_{k}\in C_{k}$, $x_{k}\gamma=a_{k}$ for any $a_{k}\in A_{k}$ ($1\leq k \leq p$), in particular $x_{k}\gamma=t_{k}$, where $t_{k}\in T_{\alpha}$. Thus, $x_{i}\gamma=t_{i}$ and $x_{j}\gamma=t_{j}$ where $t_{i}, t_{j}\in T_{\alpha}$. Observe that using \eqref{reg}, $\left|x_{i}\gamma-x_{j}\gamma\right|=\left|t_{i}-t_{j}\right|>\left|t_{i}\alpha-t_{j}\alpha\right|=\left|x_{i}-x_{j}\right|.$ This contradict the fact that $\gamma$ is a contraction, and hence the result follows.

Conversely, suppose there exists an admissible transversal $T_{\alpha}$ such that  $\left|t_{j}-t_{i}\right|=\left|t_{j}\alpha-t_{i}\alpha\right|$ for all $t_{j}, t_{i}\in T_{\alpha}$.  Define a map say $\gamma$, from $\im~\alpha$ to $T_{\alpha}$ by $x_{k}\gamma=t_{k}$ ($1\leq k\leq p$), the claim here is that, $\gamma$ is an isometry. To see this, let $x_{i}, x_{j}\in \dom~\gamma$  ($i,j\in \{1,2,\ldots,p\}$).  Then, $\left|x_{i}\gamma-x_{j}\gamma\right|=\left|t_{i}-t_{j}\right|=\left|t_{i}\alpha-t_{j}\alpha\right|=\left|x_{i}-x_{j}\right|$, as such $\gamma$ is an isometry and for any $a\in A_{k}$ ($1\leq k\leq p$), $a\alpha\gamma\alpha=x_{k}\gamma\alpha=t_{k}\alpha=a\alpha$. Thus $\alpha$ is regular.
\end{proof}

As a consequence of the above theorem, we give the following definition. An admissible transversal $T_{\alpha}$ is said to be \emph{good} if there is an isometry from $T_{\alpha}$ to $\im~\alpha$. Thus, an element $\alpha$ in $\mathcal{CP}_{n}$ is regular if and only if $\alpha$ has a good transversal. Moreover, it is not difficult to see that every convex transversal is good. We conclude the section with the following (now) obvious result:

\begin{corollary}
The semigroup $\mathcal{CP}_{n}$ ($n\geq 3$) is not regular.
\end{corollary}

\section{Green's Relations for the semigroup $\mathcal{CP}_{n}$}
 Let $S$ be a semigroup and $a,b\in S$. If $S^{1}a=S^{1}b$ (i. e., $a$ and $b$ generate the same principal left ideal) then we say that $a$ and $b$ are related by $\mathcal{L}$ and we write $(a,b)\in \mathcal{L}$ or $a\mathcal{L}b$, if $aS^{1}=bS^{1}$ (i. e., $a$ and $b$ generate the same principal right ideal) then we say $a$ and $b$ are related by $\mathcal{R}$ and we write $(a,b)\in \mathcal{R}$ or $a\mathcal{R}b$ and if $S^{1}aS^{1}=S^{1}bS^{1}$ (i. e., $a$ and $b$ generate the same principal two sided ideal) then we say $a$ and $b$ are related by $\mathcal{J}$ and we write $(a,b)\in \mathcal{J}$ or $a\mathcal{J}b$. Each of the relations $\mathcal{L}$, $\mathcal{R}$ and $\mathcal{J}$ is an equivalence on $S$. The relations $\mathcal{H}= \mathcal{L}\cap \mathcal{R}$ and $\mathcal{D}=\mathcal{L}\circ \mathcal{R}$ are also equivalences on $S$. These five equivalences are known as Green's relations, first introduced by J. A. Green in 1951 \cite{gr}.

 To begin our investigation we introduce the following concept. A subset $A$ of $[n]$ is said to be \emph{translated} by an integer $e$ written as $A+e$  if $A+e=\{a+e: a\in A\}$.

 For two subsets $A$ and  $ B$  of $ [n]$, $A$ and $B$ are said to be $e-translates$ if and only if $B=A+e$ for some $e\in \mathbb{Z}$.

  Now, as in \eqref{1} let $\alpha$ and $\beta$ in $\mathcal{CP}_{n}$ be expressed as: \begin{equation}\label{2} \alpha=\left(\begin{array}{cccc}
                                                                            A_{1} & A_{2} & \ldots & A_{p} \\
                                                                            x_{1} & x_{2} & \ldots & x_{p}
                                                                          \end{array}
\right)~~ and~~ \beta=\left(\begin{array}{cccc}
                                                                            B_{1} & B_{2} & \ldots & B_{p} \\
                                                                            y_{1} & y_{2} & \ldots & y_{p}
                                                                          \end{array}
\right)~~ (p\leq n).\end{equation}

 Then $\dom~\alpha$ and $\dom~\beta$ are said to be $e-translates$ if and only if $A_{i}=B_{i}+e$ ($i\in\{1,2,\ldots,p\}$) for some $e\in \mathbb{Z}$. Similarly, $\im~\alpha$ and $\im~\beta$ are said to be $e-translates$  if and only if $\im~\beta=\im~\alpha+e$ (or $B_{i}\beta=A_{i}\alpha+e$)  for all $1\leq i\leq p$, for some $e\in \mathbb{Z}$.

Next, let $\alpha$ and $\beta\in \mathcal{CP}_{n}$ ($1\leq p\leq n$) be as defined in \eqref{2}, then we introduce further the following concept which helps towards characterizing the Green's $\mathcal{L}$-relation in $\mathcal{CP}_{n}$.

A partition $\textbf{Ker}~\gamma$ (for $\gamma \in \mathcal{P}_{n}$) is said to be a \emph{refinement} of the partition $\textbf{Ker}~\alpha$ if $\ker~\gamma\subseteq \ker~\alpha$. Thus, if $\textbf{Ker}~\gamma=\{A_{1}^{'}, A_{2}^{'}, \ldots, A_{s}^{'}\}$ and $\textbf{Ker}~\alpha=\{A_{1}, A_{2}, \ldots, A_{p}\}$ then $p\leq s.$
A refined partition $\textbf{Ker}~\gamma$ of $\textbf{Ker}~\alpha$ is said to be\emph{ maximum}, if  $\ker~\gamma\subseteq \ker~\alpha$ and every refined relation of $\ker~\alpha$ say $\ker~\theta$ is contained in $\ker~\gamma$. Moreover, if there are at least two maximal relations say $\ker~\tau_{i}$ ($i\geq 2$)  contained in $\ker~\alpha$, then $\ker~\gamma$ is maximum if $\ker~\gamma={\overset{}{\underset{i\geq 2}\cap}}\ker~\tau_{i}.$ A maximum  refined partition $\textbf{Ker}~\gamma$ of $\textbf{Ker}~\alpha$ is said to be \emph{admissible} if it has an admissible transversal.

We immediately have the following lemma:

\begin{lemma}\label{bb} For every $\alpha\in \mathcal{CP}_{n}$,  $\textbf{Ker}~\alpha$   has   a maximum finer partition say $\textbf{Ker}~\gamma$ (for some $\gamma\in \mathcal{CP}_{n}$) with an admissible transversal.

\end{lemma}
\begin{proof}
Let $\alpha\in \mathcal{CP}_{n}$ with $\textbf{Ker}~\alpha =\{A_{1}, A_{2}, \ldots, A_{p}\}$ either ordered with the usual ordering or not, so  there are two cases to consider:\\
Case i. Suppose $\textbf{Ker}~\alpha$ is ordered with the usual order. Thus, $\textbf{Ker}~\alpha =\{A_{1}< A_{2}< \ldots< A_{p}\}$. If $A_{i}=\{a_{i}\}$ for all $i\in\{2,\ldots,p-1\}$ then take $\textbf{Ker}~\gamma=\textbf{Ker}~\alpha =\{A_{1}<\{a_{2}\}<\ldots<\{a_{p-1}\}<A_{p}\}$ for some $\gamma\in \mathcal{CP}_{n}$ and observe that if we take $T_{\gamma}=\{\max A_{1}<a_{2}<\ldots<a_{p-1}<\min A_{p}\}$, then the map $\theta=\left( \begin{array}{ccccc}
                            A_{1} & a_{2} & \ldots &a_{p-1}& A_{p} \\
                            \max A_{1} & a_{2} & \ldots & a_{p-1} & \min A_{p}
                           \end{array}
\right)$ is a contraction and hence $\textbf{Ker}~\gamma$ is admissible. Notice that any admissible finer relation of $\ker~\alpha$  is contain in the relation $\ker~\gamma$. Thus, $\textbf{Ker}~\gamma$ is maximum admissible finer partition.

If $|A_{i}|>1$ for  some $i\in \{2\ldots,p-1\},$ then let $C={\overset{p-1}{\underset{i=1}\cup}}A_{i}=\{a_{2}<\ldots<a_{s}\}$ ( for some $s>p$) and take $\textbf{Ker}~\gamma=\{A_{1}<\{a_{2}\}<\ldots<\{a_{s}\}<A_{p}\}$ for some $\gamma\in \mathcal{CP}_{n}$ and observe that if we take $T_{\gamma}=\{\max A_{1}<a_{2}<\ldots<a_{s}<\min A_{p}\}$, then the map $\theta=\left( \begin{array}{ccccc}
                            A_{1} & a_{2} & \ldots &a_{s}& A_{p} \\
                            \max A_{1} & a_{2} & \ldots & a_{s} & \min A_{p}
                           \end{array}
\right)$ is a contraction and hence $\textbf{Ker}~\gamma$ is admissible. Notice that any admissible finer relation of $\ker~\alpha$ is contain in the relation $\ker~\gamma$. Thus, $\textbf{Ker}~\gamma$ is maximum admissible partition.\\
Case ii. Suppose $\textbf{Ker}~\alpha$ is not ordered, where $\textbf{Ker}~\alpha=\{A_{1}, A_{2}, \ldots, A_{p}\}$.

If $\textbf{Ker}~\alpha$ have a convex transversal, we are done, we take $\textbf{Ker}~\gamma=\textbf{Ker}~\alpha$ and it is maximum and admissible.
If $\textbf{Ker}~\alpha$ have a relatively convex transversal which is admissible, we are also done,  take $\textbf{Ker}~\gamma=\textbf{Ker}~\alpha$ and it is maximum.

Now suppose $\textbf{Ker}~\alpha$ have no convex or admissible relatively convex transversal. If there exists a convex block $A_{k}$ of order $j$ ($j<n$) for some $k=1,\ldots,p$ in the $\textbf{Ker}~\alpha$, we partition $A_{k}$ into singleton blocks $\{a_{k_{i}}\}$ and let $\textbf{Ker}~\gamma=\{A_{1},\ldots,A_{k-1}, \{a_{k_{1}}\},\{a_{k_{2}}\},\ldots,\{a_{k_{j}}\}, A_{k+1},\ldots, A_{p}  \}$, then we can check wether $\textbf{Ker}~\gamma$ have an admissible transversal, if it does then we are done and if it does not, we then partition $A_{k-1}$ (or $A_{k+1}$) into singleton blocks, and we continue in this fashion until we get one and if all fails then at least we can partition $\textbf{Ker}~\alpha$ into singletons, and the resultant relation is the intersection of all maximal relations contain in $\ker~\alpha$ which is maximum and admissible. This complete the proof.
\end{proof}

\begin{remark}\label{re3} If $\alpha$ is a regular element in $\mathcal{CP}_{n}$, then in view of Theorem\eqref{reg}, $\textbf{Ker}~\alpha$ is a maximum admissible refinement of $\textbf{Ker}~\alpha$, since it has an admissible transversal.
\end{remark}

We now give a characterization of the Green's $\mathcal{L}$-relation  on $\mathcal{CP}_{n}$ as follows:

\begin{theorem}\label{14} Let $\alpha, \beta\in \mathcal{CP}_{n}$ be as expressed in \eqref{2}. Then $(\alpha, \beta)\in \mathcal{L}$ if and only if   $\textbf{Ker}~\alpha$ and $\textbf{Ker}~\beta$ have  admissible finer partitions,  $\textbf{Ker}~\gamma_{1}$ and $\textbf{Ker}~\gamma_{2}$ (for some $\gamma_{1}$ and $\gamma_{2}$ in $\mathcal{CP}_{n}$), respectively, such that there exists either a translation $\tau_{i}\mapsto \sigma_{i}$ and $\tau_{i}\alpha= \sigma_{i}\beta$ or a reflection $\tau_{i}\mapsto \sigma_{s-i+1}$ and $\tau_{i}\alpha= \sigma_{s-i+1}\beta$ for all $i=1,\ldots,s$ ($s\geq p$), where $A_{\alpha}=\{\tau_{1}, \ldots,\tau_{s}\}$ and $B_{\alpha}=\{\sigma_{1}, \ldots,\sigma_{s}\}$, are  the admissible transversals of $\textbf{Ker}~\gamma_{1}$ and $\textbf{Ker}~\gamma_{2}$, respectively.
\end{theorem}
\begin{proof}
Let $\alpha,\beta\in \mathcal{CP}_{n}$ be as expressed in \eqref{2} such that $(\alpha, \beta)\in \mathcal{L}$.  That is to say there exist $\gamma_{1}$, $\gamma_{2}\in(\mathcal{CP}_{n})^{1}$ such that
  \begin{equation}\label{x3} \alpha=\gamma_{1}\beta \ \ and \ \   \beta=\gamma_{2}\alpha. \end{equation}
 Let $\textbf{Ker}~\gamma_{1}=\{A_{1}^{'}, A_{2}^{'}, \ldots, A_{s}^{'}\}$ and $\textbf{Ker}~\gamma_{2}=\{B_{1}^{'}, B_{2}^{'}, \ldots, B_{s}^{'}\}$ ($s\geq p$). Whereas  $\textbf{Ker}~\alpha=\{A_{1}, A_{2}, \ldots, A_{p}\}$ and $\textbf{Ker}~\beta=\{B_{1}, B_{2}, \ldots, B_{p}\}$.  Clearly $\textbf{Ker}~\gamma_{1}$ and $\textbf{Ker}~\gamma_{2}$ are finer partitions of $\textbf{Ker}~\alpha$ and $\textbf{Ker}~\beta$, respectively.

Now let $a_{i}\in A_{i}^{'}$ and $a_{j}\in A_{j}^{'}$ ($1\leq i, j,\leq s$) where $A_{i}^{'}\subseteq A_{u}$ and $A_{j}^{'}\subseteq A_{v}$  for some  $1\leq u, v,\leq s$. Then since $\alpha=\gamma_{1}\beta$, there exists $b_{u}^{'}\in B_{u}$, $b_{v}^{'}\in B_{v}$ (for some $1\leq u,v\leq p$) such that $a_{i}\gamma_{1}=b_{u}^{'}$ and $a_{j}\gamma_{1}=b_{v}^{'}$, and that $a_{i}\gamma_{1}\beta=b_{u}^{'}\beta=a_{i}\alpha$ and $a_{j}\gamma_{1}\beta=b_{v}^{'}\beta=a_{i}\alpha$.
Since, $\gamma_{1}$ is a contraction we have; \begin{equation}\label{x1} \left|b_{u}^{'}-b_{v}^{'}\right|=\left|a_{i}\gamma_{1}-a_{j}\gamma_{1}\right|\leq \left|a_{i}-a_{j}\right|~for~all~a_{i}^{'}\in A_{i}^{'}~~and~a_{j}^{'}\in A_{j}^{'}~for~i,j\in\{1,\ldots,s\}\end{equation}

Now, recall that by Lemma\eqref{1} there exists $\tau_{i}\in A_{i}^{'}$ and $\tau_{j}\in A_{j}^{'}$ such that \begin{equation}\label{x2} \left|\tau_{i}-\tau_{j}\right|\leq \left|a_{i}-a_{j}\right|~for~all~a_{i}\in A_{i}^{'}~~and~a_{j}\in A_{j}^{'}~for~i,j\in\{1,\ldots,s\}\end{equation}

Thus, equation \eqref{x1} and \eqref{x2} implies that \begin{equation}\label{x3} \left|b_{u}^{'}-b_{v}^{'}\right|\leq\left|\tau_{i}-\tau_{j}\right|.\end{equation}

Similarly, let  $b_{i}\in B_{i}^{'}$ and $b_{j}\in B_{j}^{'}$ ($1\leq i, j,\leq s$) where $B_{i}^{'}\subseteq B_{u}$ and $B_{j}^{'}\subseteq B_{v}$ for some $1\leq u, v,\leq s$.  Then since $\beta=\gamma_{2}\alpha$, there exist $a_{u}^{'}\in A_{u}$, $a_{v}^{'}\in A_{v}$ (for some $1\leq u,v\leq p$) such that $b_{i}\gamma_{2}=a_{u}^{'}$ and $b_{j}\gamma_{2}=a_{v}^{'}$, and that $b_{i}\gamma_{2}\alpha=a_{u}^{'}\alpha=b_{i}\beta$ and $b_{j}\gamma_{2}\alpha=a_{v}^{'}\alpha=b_{i}\beta$.

Since, $\gamma_{2}$ is a contraction we have; \begin{equation}\label{x4} \left|a_{u}^{'}-a_{v}^{'}\right|=\left|b_{i}\gamma_{2}-b_{j}\gamma_{1}\right|\leq \left|b_{i}-b_{j}\right|~for~all~b_{i}^{'}\in B_{i}^{'}~~and~a_{j}^{'}\in B_{j}^{'}~for~i,j\in\{1,\ldots,s\}\end{equation}

Now, recall  that by Lemma\eqref{1} there exists $\sigma_{i}\in B_{i}^{'}$ and $\sigma_{j}\in B_{j}^{'}$ such that \begin{equation}\label{x5} \left|\sigma_{i}-\sigma_{j}\right|\leq \left|b_{i}-b_{j}\right|~for~all~b_{i}\in B_{i}^{'}~~and~b_{j}\in B_{j}^{'}~for~i,j\in\{1,\ldots,s\}\end{equation}

Thus, equation \eqref{x4} and \eqref{x5} implies that \begin{equation}\label{x6} \left|a_{u}^{'}-a_{v}^{'}\right|\leq\left|\sigma_{i}-\sigma_{j}\right|.\end{equation}

Notice that, since $\sigma_{i}\in B_{i}^{'}\subset B_{u}$ and $\sigma_{j}\in B_{j}^{'}\subset B_{v}$ then  from \eqref{x3} we have \begin{equation}\label{x7}  \left|\sigma_{i}-\sigma_{j}\right|\leq\left|b_{u}^{'}-b_{v}^{'}\right|\leq\left|\tau_{i}-\tau_{j}\right| \end{equation}
and  also  since $\tau_{i}\in A_{i}^{'}\subset A_{u}$ and $\tau_{j}\in A_{j}^{'}\subset B_{v}$ then from equation\eqref{x6} we have; \begin{equation}\label{x8}  \left|\tau_{i}-\tau_{j}\right|\leq  \left|a_{u}^{'}-a_{v}^{'}\right|\leq\left|\sigma_{i}-\sigma_{j}\right| \end{equation}
Now equation \eqref{x7} and \eqref{x8} ensure that \begin{equation}\label{x9}  \left|\tau_{i}-\tau_{j}\right|=\left|\sigma_{i}-\sigma_{j}\right|. \end{equation}

This shows that there is an isometry from $\{\tau_{i}\in A_{i}^{'}: 1\leq i\leq s\}=A_{\alpha}$ and $\{\sigma_{i}\in B_{i}^{'}: 1\leq i\leq s\}=B_{\beta}$. Now observe that  by equation \eqref{x2} and \eqref{x5} the maps $A_{i}^{'}\mapsto \tau_{i}$ and $B_{i}^{'}\mapsto \sigma_{i}$ for all $i=1,\ldots,s$ are  contractions. Thus, $A_{\alpha}$ and $B_{\beta}$ are admissible and hence $\textbf{Ker}~\gamma_{1}$ and $\textbf{Ker}~\gamma_{2}$ are admissible finer partitions of $\textbf{Ker}~\alpha$ and $\textbf{Ker}~\beta$, respectively. And equation\eqref{x9} shows  there is translation $\tau_{i}\mapsto \sigma_{i}$ or a reflection $\tau_{i}\mapsto \sigma_{s-i+1}$ ($1\leq i\leq s$). Now we claim that if $\tau_{i}\mapsto \sigma_{i}$ then $\tau_{i}\alpha= \sigma_{i}\beta$, and if $\tau_{i}\mapsto \sigma_{s-i+1}$ then $\tau_{i}\alpha= \sigma_{i}\beta$.

Now to show that $\tau_{i}\alpha=\sigma_{i}\beta$ ($1\leq i\leq s$), we suppose by way of contradiction that $\tau_{i}\alpha=x_{i}$ ($1\leq i\leq s$) and that, $\sigma_{s-u-1}\beta=x_{s-u}$,  $\sigma_{s-u}\beta=x_{s-u-1}$ and  $\sigma_{j}\beta=x_{j}$ ( $1\leq u
+1\leq j\leq s-u-2$  and $0\leq u\leq s-1$) where $\tau_{i}\in A_{\alpha}$, $\sigma_{i}\in A_{\beta}$ ($1\leq i\leq s$). Let $A_{\alpha}=\{\tau_{1}<\tau_{2}<\ldots<\tau_{s}\}$.

It is clear from \eqref{x5} that $A_{\alpha}\subseteq \dom~\gamma_{1}$ and $B_{\beta}\subseteq \dom~\gamma_{2}$. Thus, $\tau_{s-u-2}$, $\tau_{s-u-1}$, $\tau_{s-u}\in \dom~\gamma_{1}$. Notice that $\alpha=\gamma_{1}\beta$, then $\tau_{s-u-2}\gamma_{1}=b_{s-u-2}^{\prime}$, $\tau_{s-u-1}\gamma_{1}=b_{s-u}^{\prime}$ and $\tau_{s-u}\gamma_{1}=b_{s-u-1}^{\prime}$ for any $b_{s-u-2}^{\prime}\in B_{s-u-2}^{'}$, $b_{s-u-1}^{\prime}\in B_{s-u-1}^{'}$, $b_{s-u}^{\prime}\in B_{s-u}^{'}$. Thus, in particular,$\tau_{s-u-2}\gamma_{1}=\sigma_{s-u-2}$, $\tau_{s-u-1}\gamma_{1}=\sigma_{s-u}$ and $\tau_{s-u}\gamma_{1}=\sigma_{s-u-1}$.  Therefore  \begin{equation} \left|\tau_{s-u-2}\gamma_{1}-\tau_{s-u-1}\gamma_{1}\right|=\left|\sigma_{s-u-2}-\sigma_{s-u}\right|=\left|\tau_{s-u-2}-\tau_{s-u}\right|>\left|\tau_{s-u-2}-\tau_{s-u-1}\right|.\end{equation}
This contradicts the fact that $\gamma_{1}$ is a contraction, as such $\tau_{i}\alpha= \sigma_{i}\beta$ for all $1\leq i\leq s$.  Using the same argument, we can equally show that if $\tau_{i}\mapsto \sigma_{s-i+1}$ then $\tau_{i}\alpha= \sigma_{p-i+1}\beta$ for all $1\leq i\leq s$.

Conversely, suppose $\textbf{Ker}~\alpha$ and $\textbf{Ker}~\beta$ have a maximum admissible finer partitions,  $\textbf{Ker}~\gamma_{1}$ and $\textbf{Ker}~\gamma_{2}$ (for some $\gamma_{1}$ and $\gamma_{2}$ in $\mathcal{CP}_{n}$) respectively. Further, let  $A_{\alpha}=\{\tau_{1}, \ldots,\tau_{s}\}$ and $B_{\alpha}=\{\sigma_{1}, \ldots,\sigma_{s}\}$ be  the admissible transversals of $\textbf{Ker}~\gamma_{1}$ and $\textbf{Ker}~\gamma_{2}$, respectively, such that there exists either a translation $\tau_{i}\mapsto \sigma_{i}$ and $\tau_{i}\alpha= \sigma_{i}\beta$ or a reflection $\tau_{i}\mapsto \sigma_{s-i+1}$ and $\tau_{i}\alpha= \sigma_{s-i+1}\beta$ for all $i=1,\ldots,s$ ($s\geq p$)

If in the former,  $\tau_{i}\mapsto \sigma_{i}$ is  a translation and $\tau_{i}\alpha= \sigma_{i}\beta$   ($i=1,\ldots,s$), then  define $\gamma_{1}=\left( \begin{array}{cccc}
                           A_{1}^{'} & A_{2}^{'} & \ldots & A_{s}^{'} \\
                           \sigma_{1} &\sigma_{2} & \ldots & \sigma_{s}
                         \end{array}
   \right)$  and $\gamma_{2}=\left( \begin{array}{cccc}
                           B_{1}^{'} & B_{2}^{'} & \ldots & B_{s}^{'} \\
                           \tau_{1} &\tau_{2} & \ldots & \tau_{s}
                         \end{array}
   \right)$. Then $\gamma_{1}$ and $\gamma_{2}$ are contractions since for all $a_{i}^{'}\in A_{i}^{'}$, $a_{j}^{'}\in A_{j}^{'}$ ($i,j\in\{1,\ldots,s\}$) $$\left|a_{i}^{\prime}\gamma_{1}-a_{j}^{\prime}\gamma_{1}\right|=\left|\sigma_{i}-\sigma_{j}\right|=\left|\tau_{i}-\tau_{j}\right|\leq\left|a_{i}^{'}-a_{j}^{'}\right|$$ and
 for all $b_{i}^{'}\in B_{i}^{'}$, $b_{j}^{'}\in B_{j}^{'}$ ($i,j\in\{1,\ldots,s\}$) $$\left|b_{i}^{\prime}\gamma_{2}-b_{j}^{\prime}\gamma_{2}\right|=\left|\tau_{i}-\tau_{j}\right|=\left|\sigma_{i}-\sigma_{j}\right|\leq\left|b_{i}^{'}-b_{j}^{'}\right|.$$

In the later, Suppose that there exists  a reflection $\tau_{i}\mapsto \sigma_{s-i+1}$ and $\tau_{i}\alpha= \sigma_{s-i+1}\beta$ for all $i=1,\ldots,s.$

 Define $\gamma_{1}=\left( \begin{array}{cccc}
                           A_{1}^{'} & A_{2}^{'} & \ldots & A_{s}^{'} \\
                           \sigma_{s} &\sigma_{s-1} & \ldots & \sigma_{1}
                         \end{array}
   \right)$  and $\gamma_{2}=\left( \begin{array}{cccc}
                           B_{1}^{'} & B_{2}^{'} & \ldots & B_{s}^{'} \\
                           \tau_{s} &\tau_{s-1} & \ldots & \tau_{1}
                         \end{array}
   \right)$. Then $\gamma_{1}$ and $\gamma_{2}$ are contractions since for all $a_{i}^{'}\in A_{i}^{'}$, $a_{j}^{'}\in A_{j}^{'}$ ($i,j\in\{1,\ldots,s\}$) $$\left|a_{i}^{\prime}\gamma_{1}-a_{j}^{\prime}\gamma_{1}\right|=\left|\sigma_{s-i+1}-\sigma_{s-j+1}\right|\leq\left|\tau_{i}-\tau_{j}\right|\leq\left|a_{i}^{'}-a_{j}^{'}\right|$$ and
 for all $b_{i}^{'}\in B_{i}^{'}$, $b_{j}^{'}\in B_{j}^{'}$ ($i,j\in\{1,\ldots,s\}$) $$\left|b_{i}^{\prime}\gamma_{2}-b_{j}^{\prime}\gamma_{2}\right|=\left|\tau_{s-i+1}-\tau_{s-j+1}\right|\leq\left|\sigma_{i}-\sigma_{j}\right|\leq\left|b_{i}^{'}-b_{j}^{'}\right|.$$

Now by direct computations, it follows easily that $ \alpha=\gamma_{1}\beta$ and $\beta=\gamma_{2}\alpha$. Thus $(\alpha,\beta)\in \mathcal{L}$. This complete the proof.

\end{proof}
\begin{lemma}\label{v33} Let $S=\mathcal{P}_{n}$ and let $\alpha,\beta,\gamma\in S$ such that $|\im~\alpha|=|\im~\beta|$. If $\alpha=\beta\gamma$ there exists $\gamma^{'}\in S$ such that $|\im~\gamma^{'}|=|\im~\alpha|$ and $\alpha=\beta\gamma^{'}$
\end{lemma}
\begin{proof} Take $\gamma^{'}=\gamma id_{\im~\alpha}$, then $|\im~\gamma^{'}|=|\im~\alpha|$ and it is easy to see that  $\alpha=\beta\gamma^{'}$.
\end{proof}
Next let  $\alpha, \beta\in\mathcal{CP}_{n}$ be as expressed in \eqref{2}. Then we have the following:

\begin{theorem}\label{15} Let $\alpha, \beta\in \mathcal{CP}_{n}$. Then $(\alpha, \beta)\in \mathcal{R}$ if and only if $\ker~\alpha=\ker~\beta$ and there exists either a translation $x_{i}\mapsto y_{i}$ or a reflection $x_{i}\mapsto y_{p-i+1}$ ($1\leq i\leq p$).
 \end{theorem}

\begin{proof}
Suppose $(\alpha, \beta)\in \mathcal{R}$, then there exist $\gamma_{1}$, $\gamma_{2}\in (\mathcal{CP}_{n})^{1}$ such that
\begin{equation}\label{125}\alpha=\beta\gamma_{1} \ and \ \beta=\alpha\gamma_{2}.\end{equation} Thus, $\ker~\alpha=\ker~\beta$ follows easily.  Since $\alpha$ and $\beta$ have the same height, then by Lemma\eqref{v33} we can take $\gamma_{1}$ and $\gamma_{2}$ of  the same height as the height of $\alpha$ and $\beta$, and by  \eqref{125} $\im~\beta$ must be a transversal of $\textbf{Ker}\gamma_{1}$ and   $\im~\alpha$ must be a transversal of $\textbf{Ker}\gamma_{2}$.   Let $\gamma_{1}=\left(\begin{array}{cccc}
                                                                   C_{1} & C_{2} & \ldots & C_{p} \\
                                                                   x_{1} & x_{2} & \ldots & x_{p}
                                                                 \end{array}
 \right)$ and $\gamma_{2}=\left(\begin{array}{cccc}
                                                                   D_{1} & D_{2} & \ldots & D_{p} \\
                                                                   y_{1} & y_{2} & \ldots & y_{p}
                                                                 \end{array}
 \right)$ $(1\leq p\leq n)$. Thus $y_{t}\in C_{t}$ and $x_{t}\in D_{t}$ for all $1\leq t\leq p$. Since $\gamma_{1}$ and $\gamma_{2}$ are contractions, for all $i,j\in\{1,2,\ldots,p\}$ with $i<j$,  \begin{equation}\label{126} \left|x_{i}-x_{j}\right|=\left|C_{i}\gamma_{1}-C_{j}\gamma_{1}\right|=\left|y_{i}\gamma_{1}-y_{j}\gamma_{1}\right|\leq\left|y_{i}-y_{j}\right| \ \end{equation} and
\begin{equation}\label{127} \left|y_{i}-y_{j}\right|=\left|D_{i}\gamma_{2}-D_{j}\gamma_{2}\right|=\left|x_{i}\gamma_{2}-x_{j}\gamma_{2}\right|\leq\left|x_{i}-x_{j}\right|. \end{equation}
Thus from \eqref{126} and \eqref{127} we have $\left|x_{i}-x_{j}\right|=\left|y_{i}-y_{j}\right|$. This implies that there exists a translation $x_{i}\mapsto y_{i}$ or a reflection $x_{i}\mapsto y_{p-i+1}$ for all $i\in\{1,2,\ldots,p\}$.

Conversely, suppose that $\ker~\alpha=\ker~\beta$ and there exists an isometry from $\im~\alpha$ to $\im~\beta$. This implies that there exists either a translation $x_{i}\mapsto y_{i}$ or a reflection $x_{i}\mapsto y_{p-i+1}$ for all $i\in\{1,2,\ldots,p\}$.
If the map  is a translation, define $\gamma$ from $\im~\alpha$ to $\im~\beta$ by $ x_{i}\gamma=y_{i}$; and if it's a reflection, define $\gamma$ by $ x_{i}\gamma=y_{p-i+1}$ ($1\leq i\leq p$). In each case, it is easy to see that $\gamma$ is an isometry.

Now suppose that $x_{i}\gamma=y_{i}$. Let  $a \in A_{i}$ ($1\leq i \leq p$), then $a\alpha=x_{i}$ implies $a\alpha\gamma=x_{i}\gamma=y_{i}=a\beta$, as such $\alpha=\beta\gamma$.

Similarly, suppose that $x_{i}\gamma=y_{p-i+1}.$ Let  $a \in A_{i}$ ($1\leq i \leq p$), then $a\alpha=x_{i}$ implies $a\alpha\gamma=x_{i}\gamma=y_{p-i+1}=a\beta$, as such $\alpha=\beta\gamma$. Moreover, since $\gamma$ is an isometry its inverse exists, and therefore $\beta=\alpha\gamma_{1}^{-1}$. Hence $\alpha \mathcal{R} \beta$, as required.
 \end{proof}

\begin{theorem}\label{last} Let $\alpha, \beta\in \mathcal{CP}_{n}$ be as expressed in \eqref{2}. Then $(\alpha,\beta)\in \mathcal{D}$ if and only if there exist isometries $\vartheta_{1}$ from $\textbf{Ker}~\gamma_{1}$ to $\textbf{Ker}~\gamma_{2}$ and $\vartheta_{2}$ from $\im~\alpha$ to $\im~\beta$, where $\textbf{Ker}~\gamma_{1}$ and $\textbf{Ker}~\gamma_{2}$ (for some $\gamma_{1}$ and $\gamma_{2}$ in $\mathcal{CP}_{n}$) are maximum admissible finer partitions of  $\textbf{Ker}~\alpha$ and $\textbf{Ker}~\beta$, respectively.

\end{theorem}

\begin{proof}

Let $\alpha, \beta\in \mathcal{CP}_{n}$ ($1\leq p\leq n$) be as expressed in \eqref{2}.

Suppose $(\alpha, \beta)\in \mathcal{D}$.  That is to say there exists $\eta\in (\mathcal{CP}_{n})^{1}$ such that $\alpha \mathcal{L} \eta$ and $\eta \mathcal{R} \beta$. Thus, by Theorem\eqref{14}, $\alpha \mathcal{L} \eta$  implies  that there exists an isometry from the refined partition $\textbf{Ker}~\gamma_{1}$ of $\textbf{Ker}~\alpha$ to the refined partition $\textbf{Ker}~\gamma_{2}$ (for some $\gamma_{1}, \gamma_{2}\in \mathcal{CP}_{n}$) of $\textbf{Ker}~\eta$ and $\tau_{i}\alpha=\delta_{i}\eta$ or $\tau_{i}\alpha=\delta_{s-i+1}\eta$ with $\tau_{i}\in A_{\alpha}$ and $\delta_{i}\in C_{\eta}$ (where $A_{\alpha}, C_{\eta}$ denote the  admissible transversals of the maximum finer partitions $\textbf{Ker}~\gamma_{1}$ and $\textbf{Ker}~\gamma_{2}$, respectively). This implies that $\im~\alpha=\im~\eta.$
Furthermore, by Theorem\eqref{15} $\eta \mathcal{R} \beta$  implies $\ker~\eta=\ker~\beta$, i. e.,  $\textbf{Ker}~\eta=\textbf{Ker}~\beta$ and there exists an isometry from $\im~\eta$ to $\im~\beta$.  Now since $\textbf{Ker}~\eta=\textbf{Ker}~\beta$ it means that $\textbf{Ker}~\gamma_{2}$ is the maximum admissible refined partition of $\textbf{Ker}~\beta$. Hence there exists an isometry from $\textbf{Ker}~\gamma_{1}$ to $\textbf{Ker}~\gamma_{2}$.
Note also that, $\im~\alpha=\im~\eta$ and recall that there exists an isometry from $\im~\eta$ to $\im~\beta$, this implies that there exists an isometry from $\im~\alpha$ to $\im~\beta$.

Conversely, suppose there exists an isometry $\vartheta_{1}$ from $\textbf{Ker}~\gamma_{1}$ to $\textbf{Ker}~\gamma_{2}$ and also there exists an isometry $\vartheta_{2}$ from $\im~\alpha$ to $\im~\beta$. If $\vartheta_{2}$ is a reflection, i. e., $x_{i}\vartheta_{2}=y_{p-i+1}$ for all $1\leq i\leq p$, then define a map say $\gamma$ as: $$\gamma=\left( \begin{array}{cccc}
                           B_{1} & B_{2} & \ldots & B_{p} \\
                           x_{p} & x_{p-1} & \ldots & x_{1}
                         \end{array}
   \right).$$ Then  $\gamma$ is a contraction and it easily follows from Theorem\eqref{14} and \eqref{15} that  $\alpha \mathcal{L} \gamma$ and $\gamma \mathcal{R} \beta$. Hence $(\alpha, \beta)\in \mathcal{D}$.

If  $\vartheta_{2}$ is a translation, i. e., $x_{i}\vartheta_{2}=y_{i}$ for all $1\leq i\leq p$, then define a map say $\gamma$ as  $\gamma=\left( \begin{array}{cccc}
                           B_{1} & B_{2} & \ldots & B_{p} \\
                           x_{1} & x_{2} & \ldots & x_{p}
                         \end{array}
   \right).$ Then it is easy to see that $\gamma$ is a contraction and it  follows from Theorem\eqref{14} and \eqref{15} that  $\alpha \mathcal{L} \gamma$ and $\gamma \mathcal{R} \beta$. Hence $(\alpha, \beta)\in \mathcal{D}$.
\end{proof}

Let $\alpha$ and $\beta$ be regular elements  in $\mathcal{CP}_{n}$ and be as expressed in \eqref{2}. Then as a consequence of Theorem\eqref{14}, \eqref{15}, \eqref{last} and Remark\eqref{re3} we have:
\begin{corollary}\label{gr} Let $\alpha,\beta\in \mathcal{CP}_{n}$ be regular elements.
\begin{itemize}
\item[(i)] $(\alpha, \beta)\in \mathcal{L}$ if and only $\im~\alpha=\im~\beta$.
\item[(ii)] $(\alpha, \beta)\in \mathcal{R}$ if and only $\textbf{Ker}~\alpha=\textbf{Ker}~\beta$.
\item[(iii)] $(\alpha, \beta)\in \mathcal{D}$ if and only $x_{i}=y_{i}+e$ or  $x_{i}=y_{p-i+1}+e$  ($i=1,2,\ldots, p$) for some $e \in \mathbb{Z}$.
\end{itemize}
\end{corollary}

\section{Semigroup of order reversing partial contractions}
We recall that  a map $\alpha\in \mathcal{P}_{n}$ is said to be \emph{order preserving}  if  (for all $x,y \in \dom~\alpha$) $x\leq y$ implies $x\alpha\leq y\alpha$. The collection of all order preserving contractions of a finite chain $[n]$ is denoted by $\mathcal{OCP}_{n}=\{\alpha\in \mathcal{CP}_{n}: (for ~all~x,y\in \dom~\alpha)~x\leq y ~ implies ~ x\alpha\leq y\alpha\}$ and is a subsemigroup of $\mathcal{CP}_{n}$. In 2013, Zhao and Yang \cite{py} studied this semigroup, where they referred to our ``contractions" as ``compressions" and they characterized the Green's equivalences and gave a necessary and sufficient condition for an element to be regular. In this section, we deduce the regularity and Green's relations characterizations of this semigroup from the results already obtained for the larger semigroup $\mathcal{CP}_{n}$. However, before we do that, we establish the following crucial lemma.

\begin{lemma}
Let $\alpha=\left( \begin{array}{cccc}
                           A_{1} & A_{2} & \ldots & A_{p} \\
                           x_{1} & x_{2} & \ldots & x_{p}
                         \end{array}
   \right)  \in \mathcal{ORCP}_{n}~(1\leq p\leq n)$. Then the following statements are equivalent:\begin{itemize}
                    \item[(i)] $\max A_{1}-x_{1}=\min A_{p}-x_{p}=d$ and $A_{i}=\{x_{i}+d\}$ ($i=2,\ldots,p-1$), or $\max A_{1}-x_{p}=\min A_{p}-x_{1}=d$ and $A_{i}=\{x_{p-i+1}-d\}$ ($i=2,\ldots,p-1$);
                    \item[(ii)] $\textbf{Ker}~\alpha$ has a good transversal.
                  \end{itemize}
\end{lemma}

\begin{proof} Suppose (i) holds. In the former, it means that $\textbf{Ker}~\alpha=\{A_{1}<\{x_{2}+d\}<\ldots<\{x_{p-1}+d\}< A_{p}\}$. Take $T_{\alpha}=\{\max A_{1}<x_{2}+d<\ldots<x_{p-1}+d<\min A_{p}\}$. Then clearly $T_{\alpha}$ is a relatively convex transversal of $\textbf{Ker}~\alpha$ and the map $\theta=\left( \begin{array}{ccccc}
                           A_{1} & x_{2}+d & \ldots &x_{p-1}+d & A_{p} \\
                           \max A_{1} & x_{2}+d & \ldots & x_{p-1}+d & \min A_{p}
                         \end{array}
   \right)$ is clearly a contraction. Thus, $T_{\alpha}$ is admissible. Next, define a map say $\gamma$ as; $$\gamma=\left( \begin{array}{ccccc}
                           \max A_{1} & x_{2}+d & \ldots &x_{p-1}+d & \min A_{p} \\
                            x_{1} & x_{2} & \ldots & x_{p-1} & x_{p}
                         \end{array}
   \right).$$

   Clearly, $\gamma$ is an isometry since $\gamma=\alpha|_{T_{\alpha}}$ and $d=\max A_{1}-x_{1}=(x_{i}+d)-x_{i}=\min A_{p}-x_{p}$ ($i=2,\ldots,p-1$). Hence $T_{\alpha}$ is good.

 In the latter,  $\textbf{Ker}~\alpha=\{ A_{p}<\{x_{p-1}-d\}<\{x_{p-2}-d\}<\ldots<\{x_{2}-d\}< A_{1}\}$. Take $T_{\alpha}=\{\max A_{1}<x_{p-1}-d<x_{p-2}-d<\ldots<x_{2}-d<\min A_{p}\}$. Then clearly $T_{\alpha}$ is a relatively convex transversal of $\textbf{Ker}~\alpha$ and the map $\theta'=\left( \begin{array}{ccccc}
                           A_{p} & x_{p-1}-d & \ldots &x_{2}-d & A_{1} \\
                           \max A_{1} & x_{p-1}-d & \ldots & x_{2}-d & \min A_{p}
                         \end{array}
   \right)$ is  a contraction. Thus, $T_{\alpha}$ is admissible. Next, define a map say $\gamma'$ as $$\gamma'=\left( \begin{array}{ccccc}
                          \max A_{1} & x_{p-1}-d & \ldots &x_{2}-d &\min A_{p} \\
                           x_{1} & x_{2} & \ldots & x_{p-1} &  x_{p}
                         \end{array}
   \right).$$
Clearly,  $\gamma'$ is a reflection of  $\gamma$ which is also an isometry. Thus, $T_{\alpha}$ is good, as  required.

Conversely, suppose $T_{\alpha}$ is good. This means that $T_{\alpha}=\{t_{1}, t_{2}, \ldots, t_{p}\}$ is an admissible relatively convex transversal of $\textbf{Ker}\alpha$ with $1\leq\max A_{1}=t_{1} < t_{2}<\ldots < t_{p}=\min A_{p}\leq n$  and the map $t_{i}\mapsto x_{i}$ ($1\leq i \leq p$) is an isometry. If $1\leq x_{1}< x_{2}<\ldots < x_{p}\leq n$ then $\min A_{p}-\max A_{1}=\left|\min A_{p}-\max A_{1}\right|=\left|t_{p}-t_{1}\right|=\left|x_{p}-x_{1}\right|=x_{p}-x_{1}$ i. e., $\min A_{p}-\max A_{1}=x_{p}-x_{1}$ or $\min A_{p}-x_{p}=\max A_{1}-x_{1}=d.$

Notice that; $\left|t_{i}-t_{j}\right|=\left|x_{i}-x_{j}\right|$ ($i,j\in\{ 2,\ldots, p-1\}$). This means that if (without loss of generality) $i<j$ then $ t_{j}-t_{i}=x_{j}-x_{i}$ which implies  $ t_{j}-x_{j}=t_{i}-x_{i}=d$ if and only if $t_{i}=x_{i}+d$. Thus, $A_{i}=\{x_{i}+d\}$ $(2\leq i\leq p-1)$.
The same result is obtained if $ n \geq x_{1}> x_{2}>\ldots > x_{p}\geq 1.$
\end{proof}

Next, in view of the above lemma, we deduce the corresponding results for regularity of elements in the semigroups of order preserving partial  contractions and  order reversing partial contractions, $\mathcal{OCP}_{n}$  and $\mathcal{ORCP}_{n}$, respectively, from Theorem\eqref{reeg}.
\begin{corollary}\label{co2} Let $S=\mathcal{ORCP}_{n}$ and $\alpha\in S$.  If $|\im~\alpha|\geq 3$, then $\alpha$ is regular if and only if either $\min A_{p}-x_{p}=\max A_{1}-x_{1}=d$ and $A_{i}=\{x_{i}+d\}$ or  $\min A_{p}-x_{1}=\max A_{1}-x_{p}=d$ and $A_{i}=\{x_{p-i+1}+d\}$, for $i=2,\ldots,p-1$.
\end{corollary}

As a consequence of the above corollary  we have:
\begin{corollary}[\cite{py}, Theorem 2.2 (3)]\label{co1} Let $S=\mathcal{OCP}_{n}$ and $\alpha\in S$.  If $|\im~\alpha|\geq 3$, then $\alpha$ is regular if and only if $\min A_{p}-x_{p}=\max A_{1}-x_{1}=d$ and $A_{i}=\{x_{i}+d\}$, for $i=2,\ldots,p-1$.
\end{corollary}
 We conclude the characterizations of the regular elements in $S=\mathcal{ORCP}_{n}$ with the following (now) obvious result:

\begin{corollary}
The semigroup $\mathcal{ORCP}_{n}$ ($n\geq 3$) is not regular.
\end{corollary}

Next, we deduce the characterizations of  Green's equivalences obtained in \cite{py} from our results  in the previous section. First, let us prove the following lemma.

\begin{lemma}\label{y2} Let $\alpha$, $\beta$ in $\mathcal{ORCP}_{n}$ be as expressed in \eqref{2}. Then the following statements are equivalent:
\begin{itemize}
  \item[(i)] Let $\max A_{1}-\max B_{1}=d$. If $\min A_{p}-\min B_{p}=d$ and $A_{i}=B_{i}+d$ for all $i=2,\ldots,p-1$ then $\alpha$ and $\beta$ are of the  same kernel type (denoted as $\alpha^{\underline{Ker}}\beta$ in \cite{py});
  \item[(ii)] There exists a contraction from $\textbf{Ker}~\gamma_{1}=\{A_{1}<a_{2}<\ldots<a_{s}<A_{p}\}$ to $\textbf{Ker}~\gamma_{2}=\{ B_{1}<b_{2}<\ldots<b_{s}< B_{p}\}$, where $\textbf{Ker}~\gamma_{1}$ and $\textbf{Ker}~\gamma_{2}$ are the maximum admissible refined partitions of $\textbf{Ker}~\alpha$ and $\textbf{Ker}~\beta$, respectively.  Moreover, ${\overset{p-1}{\underset{i=2}\cup}} {A_{i}}=\{a_{2},\ldots,a_{s}\}$ and ${\overset{p-1}{\underset{i=2}\cup}}{ B_{i}}=\{b_{2},\ldots, b_{s}\}$ for some $s\geq p-2$  where
      $A_{\alpha}=\{\{\max A_{1}\}<a_{2}<\ldots<a_{s}<\min \{A_{p}\}\}$ and $B_{\beta}=\{\{\max B_{1}\}<b_{2}<\ldots<b_{s}<\{\min B_{p}\}\}$ are  admissible transversals of $\textbf{Ker}~\gamma_{1}$ and $\textbf{Ker}~\gamma_{2}$, respectively.
\end{itemize}
\end{lemma}

\begin{proof} Suppose (i) holds. Define a map $\gamma$ from $\textbf{Ker}~\gamma_{1}$ to $\textbf{Ker}~\gamma_{2}$ by
$$a\gamma=\left\{\begin{array}{ll}
                 \max B_{1}, & \hbox{if $a\in A_{1}$;} \\
                 a- d, & \hbox{if $a= a_{i},~i\in\{2,\ldots,s-2\}$;} \\
                 \min B_{p}, & \hbox{if $a\in A_{p}$.}
               \end{array}
\right.$$

Clearly $\gamma$ is well defined and we claim that $\gamma$ is a contraction. To see this, let $x,y\in \dom~\gamma= A_{1}\cup \{a_{2}\}\cup\ldots\cup \{a_{s-2}\}\cup A_{p}$. Then there are five cases to consider:\\
Case i: If $x=y$ then $\left|x\gamma-y\gamma\right|=\left|x\gamma-x\gamma\right|=0=\left|x-x\right|\leq\left|x-y\right|$.

Case ii: If $x\in A_{1}$ and $y= a_{i}$ ($2\leq i\leq s-2$) then $\left|x\gamma-y\gamma\right|=\left|\max B_{1}-(y-d)\right|=\left|y-(\max B_{1}+d)\right|=\left|y-\max A_{1}\right|\leq\left|y-x\right|=\left|x-y\right|$.

Case iii: If $x\in A_{1}$ and $y\in A_{p}$ then $\left|x\gamma-y\gamma\right|=\left|\max B_{1}-\min B_{p}\right|=\left|(\max A_{1}-d)-(\min A_{p}-d)\right|=\left|\max A_{1}-\min A_{p}\right|\leq\left|x-y\right|$.

Case iv: If $x= a_{i}$ ($2\leq i\leq s-2$) and $y\in A_{p}$ then $\left|x\gamma-y\gamma\right|=\left|(x-d)-\min B_{p}\right|=\left|(\min B_{p}+d)-x\right|=\left|\min A_{p}-x\right|\leq\left|y-x\right|$.

Case v: If $x= a_{i}$, $y=a_{j}$ ($i,j\in\{2,\ldots, s-2\}$) then $\left|x\gamma-y\gamma\right|=\left|(x-d)-(y-d)\right|\leq\left|x-y\right|$.
Thus, $\gamma$ is a contraction.

Conversely, suppose there exists a contraction from $\textbf{Ker}~\gamma_{1}=\{A_{1}<a_{2}<\ldots<a_{s-2}<A_{p}\}$ to $\textbf{Ker}~\gamma_{2}=\{ B_{1}<b_{2}<\ldots<b_{s}< B_{p}\}$, where $\textbf{Ker}~\gamma_{1}$ and $\textbf{Ker}~\gamma_{2}$ are the maximum admissible refined partitions of $\textbf{Ker}~\alpha$ and $\textbf{Ker}~\beta$, respectively. Moreover, ${\overset{s-2}{\underset{i=2}\cup}} {A_{i}}=\{a_{2},\ldots,a_{s-2}\}$ and ${\overset{s-2}{\underset{i=2}\cup}}{ B_{i}}=\{b_{2},\ldots, b_{s-2}\}$ for some $s\geq p-2$  where $A_{\alpha}=\{\{\max A_{1}\}<a_{2}<\ldots<a_{s-2}<\min \{A_{p}\}\}$ and $B_{\beta}=\{\{\max B_{1}\}<b_{2}<\ldots<b_{s-2}<\{\min B_{p}\}\}$ are admissible transversals of $\textbf{Ker}~\gamma_{1}$ and $\textbf{Ker}~\gamma_{2}$, respectively. Further, let $\max A_{1}-\max B_{1}=d=\min A_{p}-\min B_{p}$. Thus, the map defined as $$\theta=\left( \begin{array}{ccccc}
                           \max A_{1} & a_{2} & \ldots & a_{s-2} & \min A_{p} \\
                           \max B_{1} & b_{2} & \ldots & b_{s-2} & \min B_{p}
                         \end{array}
   \right)$$ is a translation. Hence, $a_{i}=b_{i}+d$ for all $2\leq i\leq s-2$. Note that $\max A_{1}-\max B_{1}=d=\min A_{p}-\min B_{p}$. This shows that,  $\gamma_{1}$ and $\gamma_{2}$ are of the same kernel type which implies that $\alpha$ and $\beta$ are of same kernel type.
\end{proof}
In view of the above result, we deduce the following corollaries to Theorems \eqref{14}, \eqref{15}, \eqref{last} and \eqref{gr}, respectively.

\begin{corollary}\label{l1} Let $\alpha, \beta\in \mathcal{ORCP}_{n}$. Then $(\alpha, \beta)\in \mathcal{L}$ if and only if $\im~\alpha=\im~\beta$ and $\alpha^{\underline{Ker}}\beta$.

\end{corollary}
\begin{proof}
The result for elements in $\mathcal{ORCP}_{n}$ of height $1$ is obvious. Thus, without loss of generality
we may suppose that $|\im~\alpha|\geq 2$ and let $(\alpha, \beta)\in \mathcal{L}$ in $\mathcal{ORCP}_{n}$. Notice that $\textbf{Ker}~\alpha=\{A_{1}<\ldots<A_{p}\}$ and $\textbf{Ker}~\beta=\{B_{1}<\ldots<B_{p}\}$. It follows from Theorem\eqref{14}(i)   that   there exists  a translation (if $\alpha$ and $\beta$ are order preserving)   $A_{\alpha} \mapsto B_{\beta}$ and $t_{i}\alpha=t^{\prime}_{i}\beta$  ($1\leq i\leq s$)  or a reflection (if $\alpha$ or $\beta$ is order reversing) $A_{\alpha} \mapsto B_{\beta}$ and $t_{i}\alpha=t^{\prime}_{p-i+1}\beta$  ($1\leq i\leq s$)  where $t_{i}\in A_{\alpha}$, $t_{i}^{\prime}\in B_{\beta}$, and $A_{\alpha}$, $B_{\beta}$  are admissible transversals of the refined partitions of  $\textbf{Ker}~\alpha$ and $\textbf{Ker}~\beta$, respectively. Thus, by Lemma\eqref{y2} we have $\alpha^{\underline{Ker}}\beta$ and $t_{i}\alpha=t^{'}_{i}\beta$ or $t_{i}\alpha=t^{\prime}_{p-i+1}\beta$ ($1\leq i \leq s$) implies $\im~\alpha=\im~\beta$.

Conversely, suppose that $\im~\alpha=\im~\beta$ and $\alpha^{\underline{Ker}}\beta$. Note that if $\im~\alpha=\{x_{1}<x_{2}<\ldots<x_{p}\}=\{y_{1}<y_{2}<\ldots<y_{p}\}=\im~\beta$ then $x_{i}=y_{i}$ for all $i=1,\ldots,p$.
 From Lemma\eqref{y2}, we see that $\alpha^{\underline{Ker}}\beta$ implies there exists a contraction from the maximum admissible refined partitions  $\textbf{Ker}~\gamma_{1}$ to $\textbf{Ker}~\gamma_{2}$ (for some $\gamma_{1}$, $\gamma_{2}\in\mathcal{ CP}_{n}$) of  $\textbf{Ker}~\alpha$ to $\textbf{Ker}~\beta$, respectively and  that $A_{i}\alpha=B_{i}\beta$ for all $i=1,\ldots,p$.

 Furthermore, if $\im~\alpha=\{x_{1}<x_{2}<\ldots<x_{p}\}=\{y_{1}>y_{2}>\ldots>y_{p}\}=\im~\beta$ then $x_{i}=y_{p-i+1}$ for all $i=1,\ldots,p$.
 From Lemma\eqref{y2}, we see that $\alpha^{\underline{Ker}}\beta$ implies there exists a contraction from the maximum admissible refined partitions  $\textbf{Ker}~\gamma_{1}$ to $\textbf{Ker}~\gamma_{2}$ (for some $\gamma_{1}$, $\gamma_{2}\in\mathcal{ CP}_{n}$) of  $\textbf{Ker}~\alpha$ to $\textbf{Ker}~\beta$, respectively and  that $A_{i}\alpha=B_{p-i+1}\beta$ for all $i=1,\ldots,p$.

 Thus in any case by Theorem\eqref{14}(i)   $(\alpha, \beta)\in \mathcal{L}$.

\end{proof}

Next, let  $\alpha,\beta\in\mathcal{ORCP}_{n}$  be  as expressed in \eqref{2}. Then we have the following:
\begin{corollary}\label{l2} Let $\alpha, \beta\in \mathcal{ORCP}_{n}$. Then
$(\alpha, \beta)\in \mathcal{R}$ if and only if $\ker~\alpha=\ker~\beta$ and there exists a translation $x_{i}\mapsto y_{i}$ or a reflection $x_{i}\mapsto y_{p-i+1}$ for all $1\leq i\leq p.$

\end{corollary}
\begin{proof}
The result follows directly from Theorem\eqref{14}(ii).
\end{proof}
\begin{corollary}\label{l3} Let $\alpha$ and $\beta\in \mathcal{ORCP}_{n}$. Then
  $(\alpha,\beta)\in \mathcal{D}$ if and only if there exist  translations (or  reflections) $\vartheta_{1}$ and  $\vartheta_{2}$ from $\textbf{Ker}~\gamma_{1}$ to $\textbf{Ker}~\gamma_{2}$ and from $\im~\alpha$ to $\im~\beta$, respectively.
\end{corollary}

\begin{proof} It follows directly from Theorem\eqref{last}.
\end{proof}

In view of the above result, we deduce the following corollaries to Corollaries \eqref{l1}, \eqref{l2} and \eqref{l3}, respectively.

\begin{corollary}[\cite{py}, Theorem 3.1] Let $\alpha, \beta\in \mathcal{OCP}_{n}$. Then
$(\alpha, \beta)\in \mathcal{L}$ if and only if $\im~\alpha=\im~\beta$ and $\alpha^{\underline{Ker}}\beta$.
\end{corollary}

Next, let  $\alpha, \beta \in\mathcal{OCP}_{n}$  be  as expressed in \eqref{2}. Then we have  the following:

\begin{corollary}[\cite{py}, Theorem 3.2] Let $\alpha, \beta\in \mathcal{OCP}_{n}$.
 Then $(\alpha, \beta)\in \mathcal{R}$ if and only if $\ker~\alpha=\ker~\beta$ and there exists a translation $x_{i}\mapsto y_{i}$  for all $1\leq i\leq p$.
 \end{corollary}

\begin{corollary}[\cite{py} Theorem 3.3] Let $\alpha$ and $\beta\in \mathcal{OCP}_{n}$. Then
  $(\alpha,\beta)\in \mathcal{D}$ if and only if there exist translations $\vartheta_{1}$ from $\textbf{Ker}~\gamma_{1}$ to $\textbf{Ker}~\gamma_{2}$ and  $\vartheta_{2}$ from $\im~\alpha$ to $\im~\beta$.
\end{corollary}

\begin{proof} It follows directly from Corollary\eqref{l3} together with the fact that we only need the existence of translations  from $\textbf{Ker}~\gamma_{1}$ to $\textbf{Ker}~\gamma_{2}$ and from $\im~\alpha$ to $\im~\beta$.
\end{proof}

\begin{corollary}[\cite{py}, Theorem 3.5] Let $\alpha$ and $\beta\in \mathcal{OCP}_{n}$ be regular elements. Then  $(\alpha,\beta)\in \mathcal{D}$ if and only if  $x_{i}=y_{i}+e$   ($i=1,2,\ldots, p$) for some $e \in \mathbb{Z}$.
\end{corollary}

\noindent {\bf Acknowledgements.} The second named author would like to thank Bayero University and TET Fund for financial support. He would also like to thank The Petroleum Institute, Khalifa University of Science and Technology for hospitality during his 3-month research visit to the institution.


\begin{thebibliography}{99}
\markboth{Reference}{}

\bibitem{k} Al-Kharousi, F.  Kehinde, R.  and Umar, A.  On the semigroup of partial isometries of
a finite chain. \emph{Comm. Algebra} \textbf{44} (2016), 639--647.
\bibitem{adu} Adeshola, A. D. and Umar, A. Combinatorial results for certain semigroups of order-preserving full contraction mappings of a finite chain. \emph{J. Comb. Maths. and Comb. Computing}. To appear.
\bibitem{p} Catarino, P. M.  and  Higgins, P. M. The monoid of orientation-preserving mappings on a chain. \emph{Semigroup Forum } \textbf{58 }(1999), no. 2, 190--206.
\bibitem{gr}  Green, J. A. On the structure of semigroups, \emph{Ann. of Math.} (2) \textbf{54 }(1951), 163--172.
\bibitem{Maz}  Ganyushkin, O.  and Mazorchuk, V.   C\emph{lassical Finite Transformation Semigroups}. Springer$-$Verlag: London Limited (2009).
\bibitem{ph} Higgins, P. M.  \emph{Techniques of semigroup theory.} Oxford university Press (1992).
\bibitem{howi}  Howie, J. M. \emph{Fundamental of semigroup theory}. London Mathematical Society, New series 12. \emph{The Clarendon Press, Oxford University Press}, 1995.
\bibitem{kd} Magill, K. D. J.  and  Subbiah, S. Green's relations for regular elements of semigroups of endomorphisms. \emph{Canad. J. Math.} \textbf{26} (1974), 1484--1497.
\bibitem{gg} Magill, K. D. Jr.  and Subbiah, S.  Green's relations for regular elements of sandwich semigroups. I. General results. \emph{Proc. London Math. Soc.} (3) \textbf{31} (1975), no. 2, 194--210.
\bibitem{ggg} Magill, K. D. Jr.  and Subbiah, S.  Green's relations for regular elements of sandwich semigroups. II. Semigroups
of continuous functions, \emph{J. Austral. Math. Soc. Ser. A }\textbf{25} (1978), no. 1, 45--65.
\bibitem{su} Pei, H. S., Sun, L.  and Zhai, H. C.  Green's relations for the variants of transformation
semigroups preserving an equivalence relation.\emph{ Comm. Algebra} \textbf{35} (2007), no. 6, 1971--1986.
\bibitem{zou} Pei, H. S.  and Zou, D. Y.  Green's equivalences on semigroups of transformations preserving order and an equivalence.\emph{ Semigroup Forum} \textbf{71} (2005), no. 2, 241--251.
\bibitem{pe} Pei, H. S.  Regularity and Green's relations for semigroups of transformations that preserve an equivalence, \emph{Comm. Algebra} \textbf{33} (2005), no. 1, 109--118.
\bibitem{sp} Sun, L.  and Pei, H. S.  Green's relations on semigroups of transformations preserving two
equivalence relations. \emph{ J. Math. Res. Exposition} \textbf{29} (2009), no. 3, 415--422.
\bibitem{sz} Sun, L.,  Pei, H. S., and Cheng, Z. X.  Regularity and Green's relations for semigroups of
transformations preserving orientation and an equivalence. \emph{Semigroup Forum} \textbf{74} (2007),
no. 3, 473--486.
\bibitem{af} Umar, A.  and Al-Kharousi, F. Studies in semigroup of contraction mappings of a finite chain. The Research Council of Oman Research grant proporsal No. ORG/CBS/12/007, 6th March 2012.
\bibitem{ua} Umar, A.  On the semigroups of order-decreasing finite full transformations. \emph{Proc. Roy.
Soc. Edinburgh Sect. A} \textbf{120} (1992), no. 1-2, 129--142.
\bibitem{py} Zhao, P.  and Yang, M.  Regularity and Green's relations on semigroups of transformation preserving
order and compression.\emph{ Bull. Korean Math. Soc.} \textbf{49} (2012), No. 5,  1015--1025.


\end{thebibliography}
\end{document}